\documentclass[0pt,reqno]{amsart}
\usepackage{pgf,tikz}
\usepackage{mathrsfs}
\usetikzlibrary{arrows}
\usepackage{amssymb}
\usepackage{graphicx}
\usepackage{amsmath}
\usepackage{centernot}
\usepackage{blindtext}
\usepackage[english]{babel}
\usepackage[utf8]{inputenc}
\usepackage{pgf,tikz}
\usepackage{mathrsfs}
\usetikzlibrary{arrows}
\usepackage{amssymb}
\usepackage{graphicx}
\usepackage{amsmath}
\usepackage{centernot}
\usepackage{blindtext}
\usepackage{enumitem}
\usepackage{xcolor}
\usepackage{thmtools}
\usepackage{thm-restate}

\usepackage{hyperref}

\usepackage{cleveref}

\usepackage{ wasysym }
\usepackage{mathtools}
\usepackage{physics}
\numberwithin{equation}{section}
\usepackage{amsthm}
\usepackage{tikz-cd} 
\usepackage{mathtools}

\newcommand{\Id}{\text{Id}}
\newcommand{\R}{\mathbb{R}}

\newcommand{\N}{\mathbb{N}}

\newcommand{\C}{\mathbb{C}}

\newcommand{\eps}{\varepsilon}
\newcommand{\Z}{\mathbb{Z}}

\newtheorem{theorem}{Theorem}[section]
\newtheorem{corollary}[theorem]{Corollary}
\newtheorem{Lemma}[theorem]{Lemma}
\newtheorem{Definition}{Definition}[section]

\usepackage[a4paper, total={6in, 9in}]{geometry}
\usepackage[utf8]{inputenc}
\usepackage[normalem]{ulem}

\title{The Minimal Denominator Function and Geometric Generalizations}
\author{Albert Artiles
\\
\small University of Washington, Seattle, USA \\
}

\date{\today}
\begin{document}
\maketitle

\begin{abstract}
\noindent We provide a geometric interpretation for a  normalized version of the minimal denominator function,
$$q_{\min}(x,\delta)=\min\left\{q\in \N: \text{ there exists } p\in\Z \text{ such that } \frac{p}{q}\in (x-\delta,x+\delta)\right\},$$
introduced by Chen and Haynes in~\cite{CH}. We use this interpretation to compute the limiting distribution of a suitably normalized version of $q_{\min}(x,\delta)$ as a function of $x$, and give generalizations of the idea of minimal denominators to higher-dimensional unimodular lattices, linear forms, and translation surfaces. The key idea is to turn this circle of problems into equidistribution problems for translates of unipotent orbits of a Lie group action on an appropriate moduli space.
\end{abstract}


\section{Introduction}

The study of the distribution of rational numbers in $\R$ has been a subject of interest for many decades. A well-known result due to Dirichlet gives quantitative information of the distribution of rational numbers with small denominators: Given $x\in [0,1]$ and $Q\geq1$, there exists $p,q\in\Z$ such that $q\leq Q$ and 
\begin{equation}\label{Dirichlet}
    \abs{qx-p}<\frac{1}{Q}.
\end{equation}


In this paper we explore the connection between approximations of real numbers by rational numbers with small denominators and the theory of lattices. We also provide generalizations to these ideas in the contexts of linear forms with entries in $[0,1]$ and saddle connections of translation surfaces.

\subsection{Minimal Denominators}\label{Minimal Denominators}
Inspired by the work of Sander-Weiss~\cite{SanderMeiss}, Chen-Haynes~\cite{CH} defined the \emph{minimal denominator function}, $q_{\min}(x,\delta)$, which extracts the lowest denominator of a rational number contained in $(x-\delta,x+\delta)$. More precisely, for $\delta >0$,  define $q_{\min}(\cdot, \delta):[0,1]\rightarrow \N$ by 

\begin{equation}\label{mindenom function}
    q_{\min}(x,\delta)=\min\left\{q\in \N: \text{there exists } p\in \Z \text{ such that } \frac{p}{q}\in(x-\delta,x+\delta)\right\}.
\end{equation}

\paragraph*{\bf Asymptotics} A main result of~\cite{CH} is the computation of the asymptotics of the expected value of $q_{\min}(x,\delta)$ when $x$ is chosen randomly uniformly with respect to the Lebesgue measure on $[0, 1]$ and $\delta \rightarrow 0$.

\begin{equation}
\mathbb{E}_x\left[q_{\min}(x,\delta)\right]=\int_0^1 q_{\min}(x,\delta) dx.
\end{equation}

\begin{theorem}\textbf{(Chen-Haynes~\cite{CH})}\label{theorem:CH} As $\delta \rightarrow0$
\begin{equation}\label{CH}
\mathbb{E}_x\left[q_{\min}(x,\delta)\right]\sim\frac{16}{\pi^2}\frac{1}{\sqrt{\delta}}+O(\log^2(\delta)).
\end{equation}
\end{theorem}

\subsection{A geometric interpretation and generalizations}\label{sec:geom} Our key contribution is to show how to interpret an appropriately normalized version of $q_{\min}$ geometrically. This allows us to compute its limiting distribution through dynamical and geometric methods. These methods of proof in turn allow us to generalize the minimal denominator function to a variety of new contexts, including higher-dimensional Diophantine approximation and the distribution of holonomy vectors of saddle connections on translation surfaces. We first consider a higher dimensional version of our question and then specialize to our problem at hand in \S\ref{existence of limiting distribution}.
\newline

\paragraph*{\bf The space of unimodular lattices} 
 Let $X_m=SL(m,\R)/SL(m,\Z)$ denote the space of \textit{unimodular lattices} in $\R^m$ equipped with the measure $\mu_m$, arising from the Haar measure on $SL(m,\R)$ normalized so that $\mu_m$ is a probability measure on $X_m$. Define the function $F^m_1:X_{m+1}\rightarrow \R$ by

\begin{equation}
    F^m_1(\Lambda)=\min\left\{x>0: \text{ there exists } \mathbf{y}\in \R^{m} \text{ such that } \begin{bmatrix}x\\ \mathbf{y}\end{bmatrix}\in \Lambda\cap C_1\right\},
\end{equation}
where $C_1=\left\{ \begin{bmatrix}
    x&\mathbf{y}
\end{bmatrix}^T\in \R^m: x>0 \text{ and } \mathbf{y}\in \R^{m}, \norm{\mathbf{y}}_{m}<x\right\}$ and $\norm{\cdot}_k$ denotes the max norm on $\R^k$.
\newline

\paragraph*{\bf Normalized $q_{\min}$ and $F^m_\delta$} A key application of our interpretation of $q_{\min}(x,\delta)$ is the following theorem which describes the limiting distribution of the normalized function $\delta^{\frac{1}{2}}q_{\min}(x,\delta)$ in terms of the function $F^1_1(\Lambda)$, where $\Lambda$ is a random unimodular lattice chosen from $X_2$ according to the probability measure $\mu_2$.

\begin{theorem}\label{cumulative 2} Let $P$ denote the uniform probability measure on $[0,1]$. Then for every $T>0$, as $\delta\rightarrow 0$,
\begin{equation}
   P\left(\left\{x: \sqrt{\delta}q_{\min}(x,\delta)\leq T \right\}\right)\rightarrow \mu_2\left(\left\{\Lambda\in X_2: F^1_1(\Lambda)\leq T\right\}\right). 
\end{equation}
\end{theorem}

\paragraph*{\bf Generalizations} The framework developed above can be generalized in two ways. One, motivated by Diophantine approximation, is to consider higher dimensional versions of the minimal denominator function and connections to dynamics on space of unimodular lattices. Another is to consider different discrete subsets of the plane which arise from geometric constructions, in particular, we look at holonomy vectors of saddle connections on translation surfaces. We explore higher-dimensions in \S\ref{higherdimensions} and discuss saddle connections in \S\ref{Saddle Connections}.
\newline
\newline
\paragraph*{\bf A higher dimensional version of $q_{\min}$} \label{Q^m}
For $m\in \N$, $\textbf{x}\in [0,1]^m$, and $\delta>0$, we define
$$Q^m(\textbf{x},\delta)=\min\left\{q\in\N:\text{ there exists }\textbf{p}\in \Z^m \text{ such that }\norm{\textbf{x}-\frac{1}{q}\textbf{p}}_m<\delta\right\}.$$
\newline
$Q^m(\mathbf{x},\delta)$ computes the smallest least common multiple of the denominators of the rational points in  a $\delta$ neighborhood of the point $x$ in $[0,1]^m.$ In particular, $Q^m$ measures the complexity of rational points in a $\delta$ neighborhood of each point in the $m$-dimensional unit cube.
In this aspect we see that $Q^1(x,\delta)=q_{\min}(x,\delta)$. We then have the following generalization of Theorem \ref{cumulative 2}

\begin{theorem}\label{cumulative m+1} Let $P$ denote the uniform distribution on $[0,1]^m$. Then, as $\delta\rightarrow 0$,
    $$P\left(\left\{\mathbf{x}\in [0,1]^m:\delta^{\frac{m}{m+1}}Q^m(\mathbf{x},\delta)\leq T\right\}\right)\rightarrow \mu_{m+1}\left(\left\{\Lambda\in X_{m+1}: F^{m}_1(\Lambda)\leq T\right\}\right).$$
\end{theorem}

\subsection{Translation surfaces}\label{translation surfaces intro}

Another context in which our geometric interpretation for $\delta^{\frac{1}{2}}q_{\min}(\cdot,\delta)$ can be generalized is in the context of the distribution of saddle connections of translation surfaces. We will review the necessary background on translation surfaces for this paper in this section. For further background on translation surfaces, see, for example, the survey articles of Zorich~\cite{Zorich} and Hubert-Schmidt~\cite{HubSchm}.
\newline

\begin{figure}[t!]
\centering
\includegraphics[width=50mm]{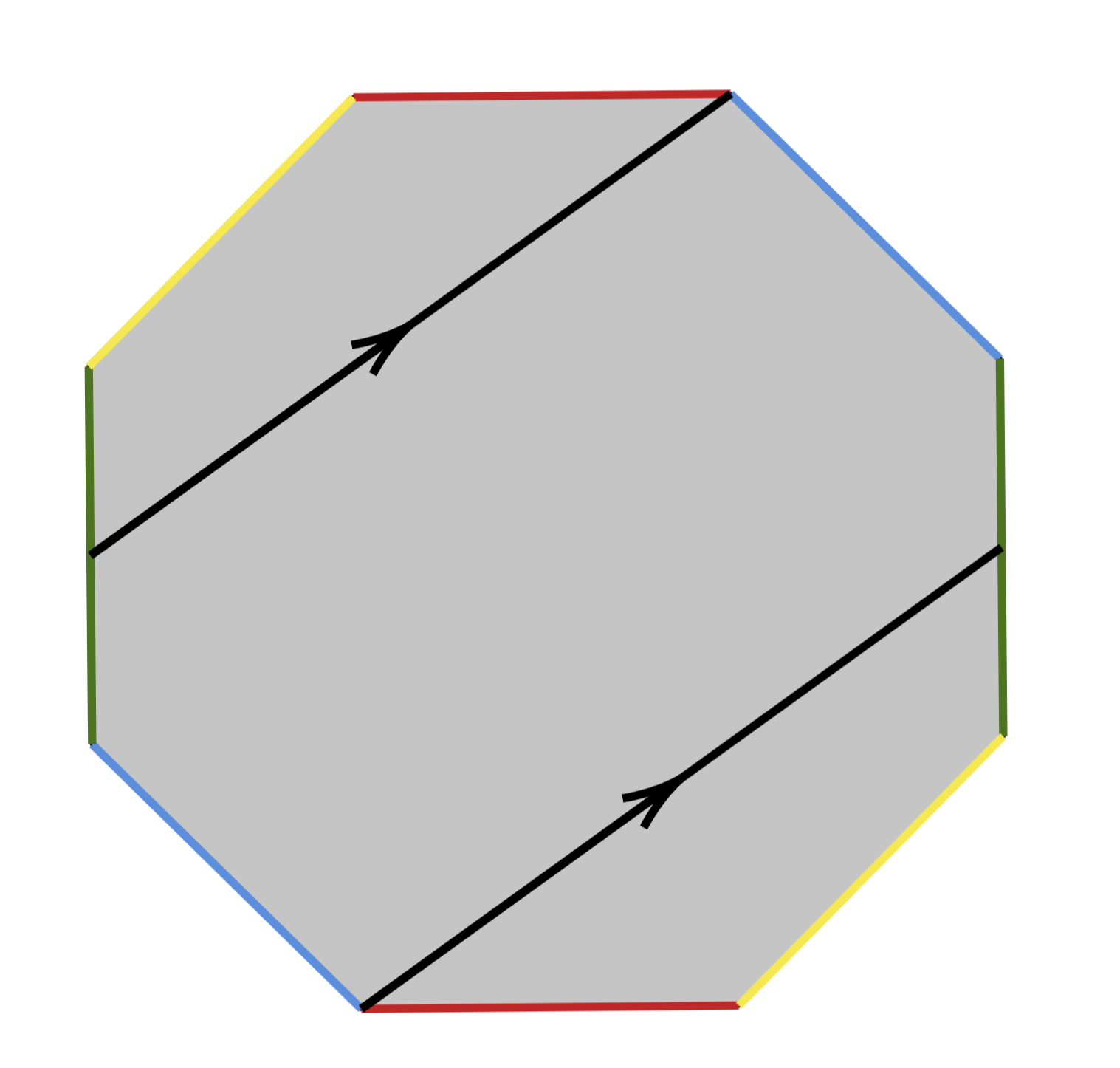}
\caption{A genus $2$ translation surface where sides are identified via color and a depiction of a saddle connection in black.} \label{Translationsurfaceandsaddleconnection}
\end{figure}

\paragraph*{\bf Translation surfaces and polygonal presentation}  A \emph{compact translation surface} is an ordered pair $(X,\omega)$ where $X$ is a compact Riemann surface and $\omega$ is a non-zero holomorphic 1-form on $X$. We usually write $\omega$ to represent the translation surface $(X,\omega)$ for convenience and write $(X,\omega)$ if clarification is needed. A geometric way to think about a translation surface is as follows: Let $P=\bigsqcup_{i=1}^{k}P_{i}$ be a disjoint union of connected polygons (not necessarily convex) $P_i\subset \C$ such that the collection of the sides can be partitioned into parallel pairs $\{s_{a},s_{b}\}$ of the same length; which are subsequently identified via the Euclidean translation $T_{a,b}$ sending $s_{a}$ onto $s_{b}$, to produce a surface. Since translations are holomorphic, the surface obtained from this procedure can be endowed with a complex structure. Moreover, since for each $c\in\C$ we have that $d(z+c) = dz$, the holomorphic $1$-form $dz$ on the plane descends to this surface and endows it with a (non-zero) holomorphic one form $\omega$.
\newline
\paragraph*{\textbf {Saddle connections and holonomy vectors.}} 

Let $\left\{z_{1},z_{2},...,z_{j}\right\}$ be the zeros of $\omega$ on $X$. $\omega$ induces a flat Riemannian metric on the surface $X\setminus \left\{z_{1},...,z_{j}\right\}$. A \emph{saddle connection} is a geodesic $\gamma$ starting and ending at two (possibly the same) zeros of $\omega$ without passing through any other zeros.



We may record saddle connections as complex numbers as follows: given a saddle connection $\gamma$ on a translation surface $\omega$, we define the $\textit{holonomy vector}$ of $\gamma$ by $$z_{\gamma}=\int_\gamma \omega \in\C.$$ The holonomy vector $$z_{\gamma} = x_{\gamma} + i y_{\gamma}$$ records the horizontal ($x_{\gamma}$) and vertical ($ y_{\gamma}$) displacement of $\gamma$ on $\omega$. Denote the collection of all holonomy vectors of $\omega$ by $$\Lambda_\omega := \{ z_{\gamma} : \gamma \mbox{ is a saddle connection on } \omega \}.$$ Masur showed in~\cite{Masur} that for any $\omega$, there exists positive constants $c_1$ and $c_2$ (only dependent on $\omega$) such that
$$c_1R^2\leq\abs{\Lambda_\omega \cap B(0,R)}\leq c_2R^2,$$
characterizing the growth rate of the $\Lambda_\omega$. Much work has been done since then to understand the distribution of saddles for different kinds of translation surfaces.
\newline
\paragraph*{\bf{Short Saddle Connections}}  Let $\omega$ be a translation surface and $\delta>0$ and consider the following function: $$\Psi(\omega,\delta)=\min\left\{\Re(z_\gamma): \gamma \text{ is a saddle connection of } \omega \text{ and } z_\gamma\in \Lambda(\omega)\cap C_\delta\right\},$$
where $C_\delta=\{x+iy\in \C: x>0 \text{ and } y<\delta x\}$.

The next result computes the limiting distribution of $\sqrt{\delta}\Psi(\omega,\delta)$ as $\delta\rightarrow0$ whenever $\omega$ is a Veech surface with Veech group $\Gamma_\omega$. We also set the notation that $Y_\omega=SL(2,\R)/\Gamma_\omega$ and $\mu_\omega$ is the induced measure on $Y_\omega$ by the Haar measure on $SL(2,\R)$. Denote the matrix $\begin{bmatrix}
    1 & 0\\ -\alpha & 1
\end{bmatrix}$ by $h_\alpha$. We then have the following result on the distribution of short saddle connections.

\begin{theorem}\label{cummulative sc} Let $\omega$ be a Veech surface and suppose that $h_{\alpha}\in \Gamma_\omega$ for some $\alpha>0$. Let $P$ be the uniform probability measure on $[0,\alpha]$. Then as $\delta\rightarrow 0$,
    $$P\left(\left\{s\in [0,\alpha]:\sqrt{\delta}\Psi(h_s\omega,\delta)\leq T\right\}\right)\rightarrow \frac{\mu_{\omega}\left(\left\{ g\Gamma_\omega\in Y_\omega: \Psi(g\Gamma_\omega,1)\leq T\right\}\right)}{\mu_\omega\left(Y_\omega\right)}.$$
\end{theorem}



\subsection{History}
The contents of this paper create a link between number theory, the theory of homogeneous dynamics, and translation surfaces. We detail the connections more explicitly below.
\newline

\paragraph*{\bf Minimal denominators}
In the 1920s Franel~\cite{Franel} and Landau~\cite{Landau} restated the Riemann Hypothesis as a problem on the distribution of the Farey Sequence in $[0,1]$. This contributed to the growth in interest of questions on the distribution of rational numbers with small denominators. For instance, Hall computed the distribution of the spacing between consecutive Farey fractions, when properly normalized~\cite{Hall}. More recently, Boca-Zaharescu computed correlation formulas for the Farey sequence in~\cite{BZ}. Theorem~\ref{cumulative 2} describes the distributions of the waiting time to for the Farey Sequence to intersect a randomly chosen small interval in $[0,1]$, expanding on the work of Chen-Haynes in~\cite{CH}.
\newline

\paragraph*{\bf Short lattice vectors}\label{shortvectors}
Given a lattice $\Lambda$ in $\R^n$ and a  norm, $\norm{\cdot}$, on $\R^n$, we may ask the following question: What is the shortest non-zero vector in $\Lambda$? This question is known as the \emph{short vector problem} (SVP) and it has been of interest in cryptography. This problem is NP-hard. We create a dictionary that allows us to connect the minimal denominator function to the minimization of the lengths of vectors inside a thin cone. Other variants of the SVP have been studied in the past. The work of Siegel~\cite{Siegel} allows us to compute the average number of intersections between a randomly chosen unimodular lattice and a region of the $\R^n$. More recently, Kim~\cite{Kim} computed the distribution for the lengths of the the first $k$-short vectors in a randomly chosen lattice.
\newline


\paragraph*{\bf Holonomy vectors}
Translation surfaces arise naturally from problems in rational billiards, Riemann surfaces, and number theory. Masur~\cite{Masur} characterized the growth rate of the set of their holonomy vectors in 1990. Over the past few decades a lot of work has been done to get finer statistics on the distribution of saddle connections of different kinds of translation surfaces. Athreya-Chaika~\cite{AthrChai} described the decay of smallest angle gaps between saddle connections in almost any surface. More recently, Kumanduri-Sanchez-Wang proved the existence of gap distribution for any Veech surface in~\cite{KSW}. A consequence of our setup and generalizations is found in Theorem~\ref{cummulative sc}, where we compute the limiting distribution of the length of the shortest saddle connection in a thin cone.

\subsection{Organization of the paper} \S\ref{sec:Preliminaries} contains the background information needed for the ideas used throughout the paper.
In \S\ref{existence of limiting distribution}, we explore the relation between $q_{\min}$ and the theory lattices in detail. We also provide geometric motivation for the proof of Theorem \ref{cumulative 2}. \S\ref{Siegel Measures} explores a general setup for a broad range of equidistribution results. We use this section to describe a motivating theorem for the generalizations of all results in this paper. The subsequent sections are applications of the general philosophy developed in \S~\ref{Siegel Measures} to generalizations of Theorem~\ref{cumulative 2}. \S\ref{higherdimensions} contains higher dimensional versions of Theorem~\ref{cumulative 2} in two contexts: higher dimensional lattices and linear forms. \S\ref{Saddle Connections} contains more information on translation surface and the proof of Theorem~\ref{cummulative sc}.
\newline

\paragraph*{\bf Acknowledgements}\label{Acknowlegements}
 I would like to thank my advisor, Jayadev Athreya, for introducing me to this collection of ideas, their exceptional patience, enthusiasm, and unconditional support.

\section{Preliminaries}\label{sec:Preliminaries}
In this section we introduce the relevant background to follow the ideas in this paper.

\subsection{The space of unimodular lattices} Our setup in \S\ref{existence of limiting distribution} and \S\ref{higherdimensions} will be on the space of \emph{unimodular lattices}. A unimodular lattice $\Lambda$ is a maximal discrete subgroup of $\R^m$ such that the volume of the quotient $\R^m/\Lambda$ is one. 

The space of unimodular lattices in $\R^m$ will be denoted by $X_m$. The group $SL(m,\R)$ acts on $\R^m$ via linear transformations. This transfers to a transitive action on $X_m$. Notice that the stabilizer of the unimodular lattice $\Z^m$ under the $SL(m,\R)$ action is $SL(m,\Z)$. This observation allows us to identify $SL(m,\R)/SL(m,\Z)$ with $X_m$ via $g\Z^m\leftrightarrow gSL(m,\Z).$  In  the remaining of the paper we will identify $X_m$ and $SL(m,\R)/SL(m,\Z)$ via this correspondence implicitly. We equip $SL(m,\R)$ with its Borel $\sigma$-algebra and its Haar measure $\tilde{\mu}_m$. We denote by $\mu_m$ the Haar probability measure on $X_m$. With this setup, we have that $\mu_m$ is an ergodic measure with respect to the $SL(m,\R)$ action on $X_m$.

\subsection{Geodesic and Horocyclic flows on $X_2$}
We use this section to describe two flows, the geodesic and the horocylcic flows on $X_2$. For each $s,t\in \R$, define $$h_s=\begin{bmatrix}1&0\\-s&1\end{bmatrix} \mbox{ and } g_t=\begin{bmatrix}e^{\frac{t}{2}}&0\\0&e^{-\frac{t}{2}}\end{bmatrix}.$$ Note that \begin{equation}\label{eq:conjugacy} g_t h_s g_{-t} = h_{se^{-t}} \end{equation}
\begin{Definition}
    The geodesic and horocyclic flow on $X_2$ are defined by $(t,g\Z^2)\mapsto (g_tg)\Z^2$ and $(s,g\Z^2)\mapsto(h_sg)\Z^2$, respectively.
\end{Definition}
Since the horocyclic and geodesic flows are defined by left translation by elements of $SL(2,\R)$, it follows that $\mu_2$ is preserved by the two flows.

Dani-Smillie proved in~\cite{DaniSmillie} the following result concerning the distribution of the orbits of points in $X_2$ under the horocyclic flow and characterization of Borel probability measures preserved by the horocyclic flow. 
\begin{theorem}(\textbf{Dani-Smillie}~\cite{DaniSmillie})\label{equidistribution}
    Let $\{\Lambda_i\}_{i\in \N}$ be a sequence of points in $X_2$. Suppose that $\Lambda_i$ have period $s_i$ under the horocyclic flow. Let $\nu_i$ be the uniform measure on the orbit of $\Lambda_{i}$, then if $s_i\rightarrow\infty$,
    $$\nu_i\stackrel{\ast}{\rightharpoonup} \mu_2,$$
    where the convergence above refers to weak$^*$ convergence.
\end{theorem}

Notice $h_s$-periodic lattices are precisely those with a vertical vector. In particular $\Z^2$ has period 1 with respect to the horocyclic flow. Since $g_th_sg_{-t}=h_{se^{-t}}$, it follows that $g_t\Z^2$ has period $e^{-t}$ under the horocyclic flow.

\section{Existence of limiting distribution for $q_{\min}$}\label{existence of limiting distribution}
In this section we are interested in the distribution of vectors with small slope and small horizontal part in a lattice and their relation to the minimal denominator function. To this end, we define the following family of functions on $X_2$: given $\delta>0$, $F^1_\delta:X_2\rightarrow \R$ is given by
$$F^1_\delta(\Lambda)=\min\left\{u\in \R:\text{ there exists } v\in\R \text{ such that } \begin{bmatrix}
    u\\v
\end{bmatrix}\in \Lambda\cap C_\delta\right\},$$
where $C_\delta=\left\{[x,y]^T\in\R^2:x>0, \abs{y}< \delta x \right\}.$

\begin{figure}[t!]
\centering
\includegraphics[width=95mm]{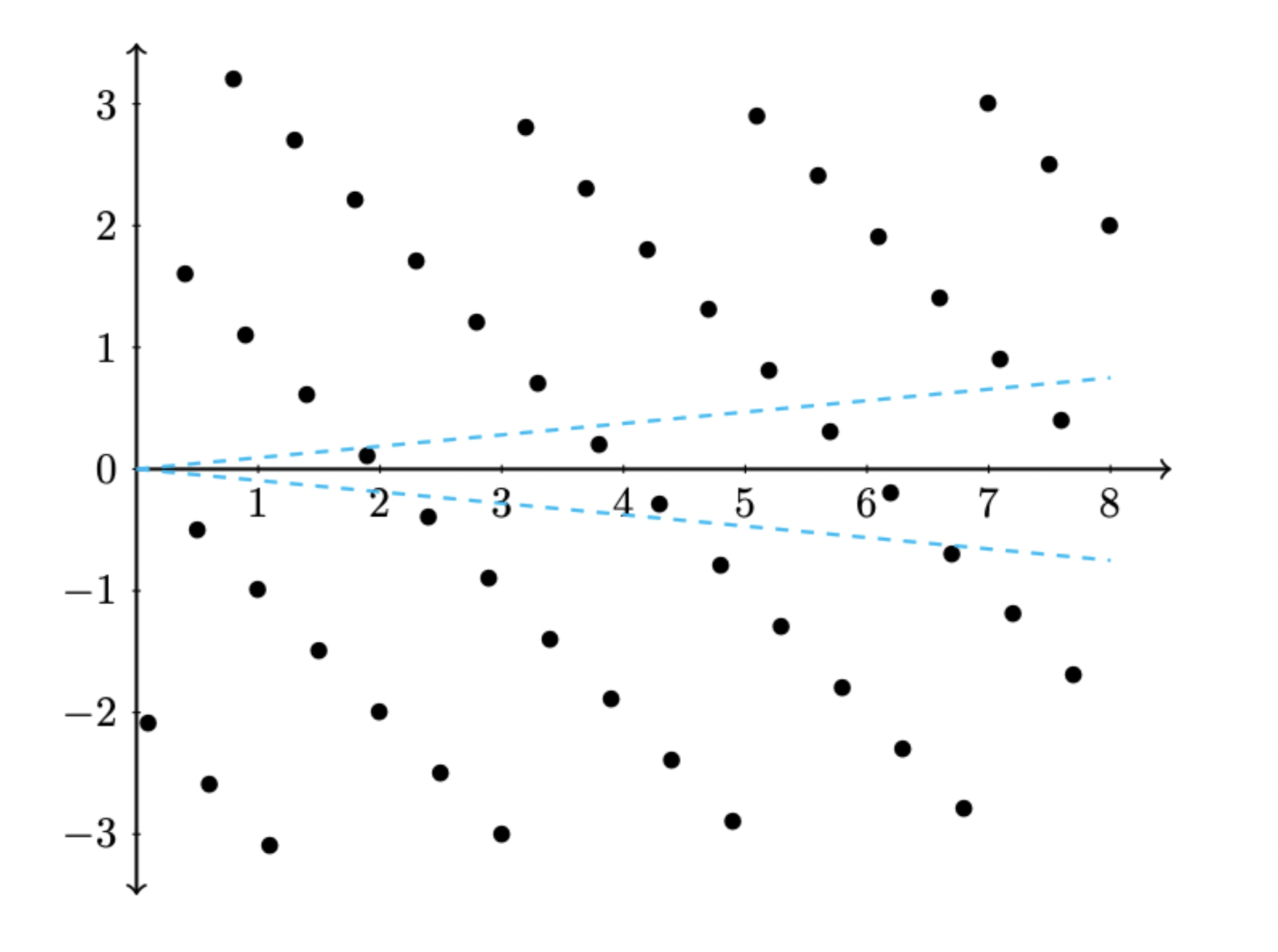}
\caption{Depiction of a unimodular lattice and its intersection with $C_\delta$.} \label{Unimodular lattice intersection cone}
\end{figure}

\begin{figure}\label{3}
\centering
\includegraphics[width=95mm]{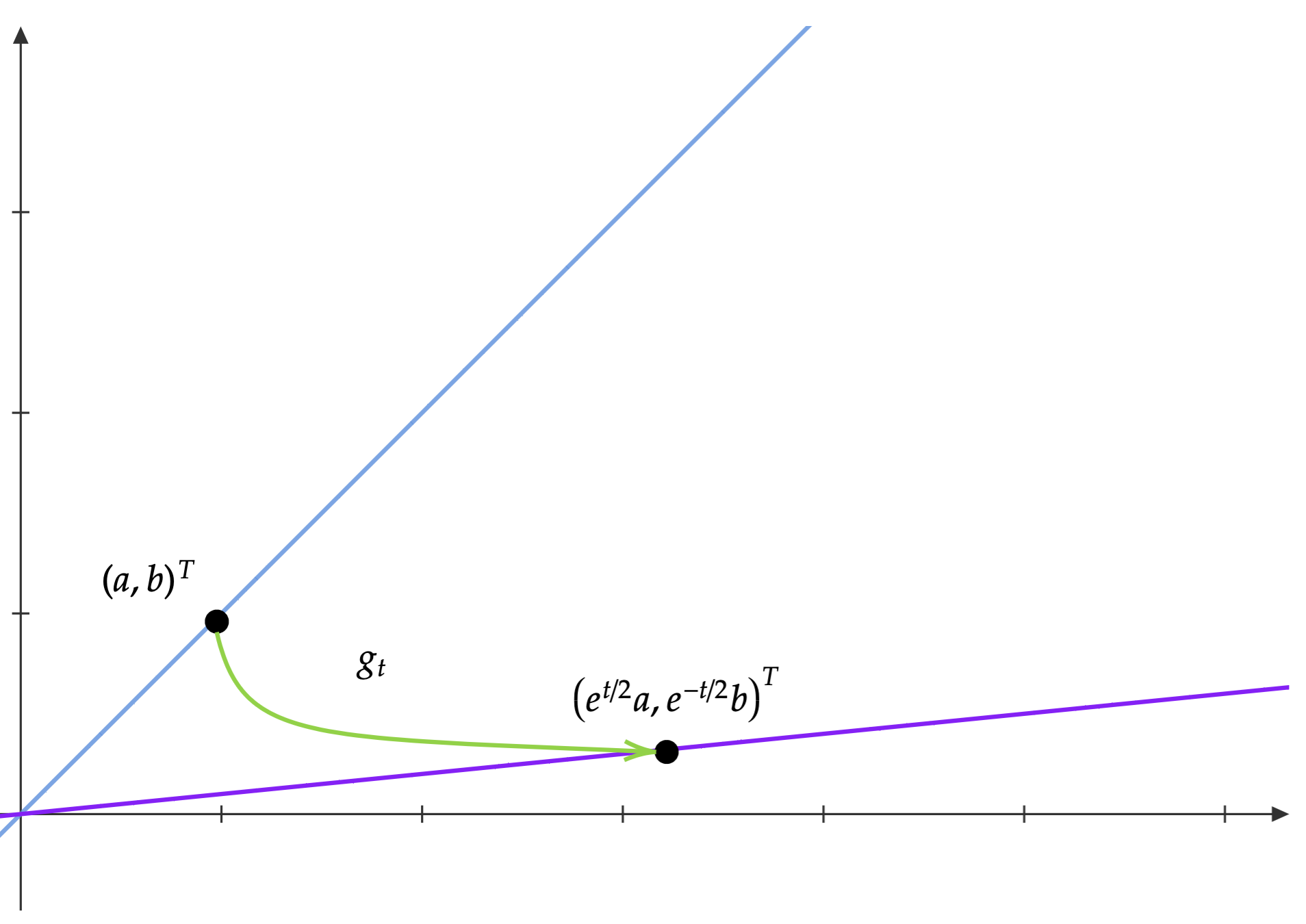}
\caption{Action of $g_{t}$ on lines though the origin.} 
\label{g_t action}
\end{figure}

\begin{figure}
\centering
\includegraphics[width=100mm]{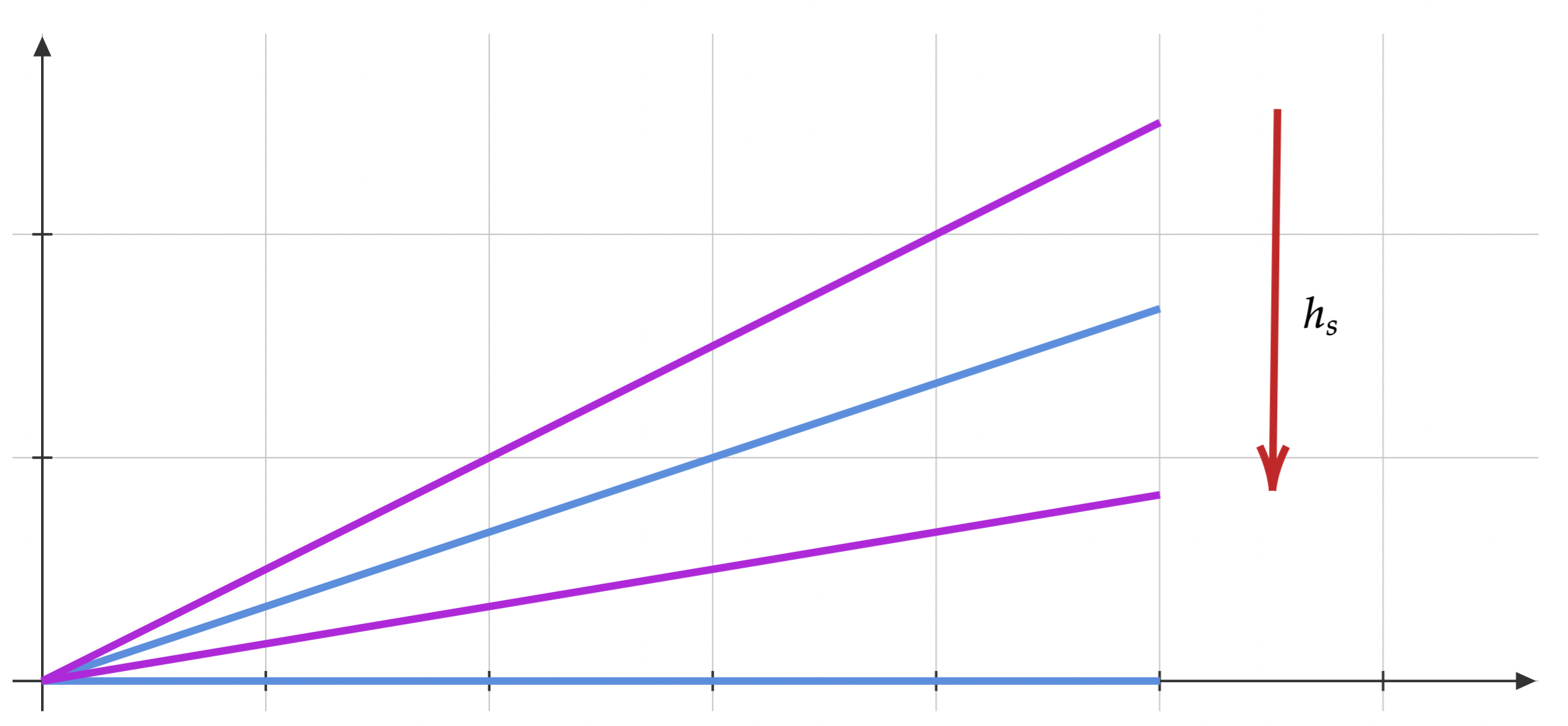}
\caption{Action of $h_{s}$ on lines though the origin.} 
\label{h_s action}
\end{figure}

$h_s$ and $g_t$ act on $\R^2$ via linear transformations and hence they act on lines through the origin. In particular, $h_s$ preserves the difference of slopes between any two lines through the origin while $g_t$ multiplies the slope of any line by $e^{-t}$.
See figures~\ref{g_t action} and~\ref{h_s action}.

We begin by working out the ways in which $F^1_\delta$ behaves when precomposed with the flows $g_t$ and $h_s$. 

\begin{Lemma}\label{good lemma} 
\textcolor{white}{fsdfa}

\begin{enumerate}
    \item  \emph{For each} $t\in \R$, \begin{equation}\label{geohomogeneity}
    F^1_\delta(g_t\Lambda)=e^{\frac{t}{2}}F^1_{\delta e^{t}}(\Lambda).
\end{equation}
 \item \emph{For each} $x\in [0,1]$, \begin{equation}\label{Fdeltaqminrelation}
    F^1_\delta(h_x\Z^2)=q_{\min}(x,\delta).
\end{equation}

\end{enumerate}
\end{Lemma}

\begin{proof}
We will prove equation (\ref{geohomogeneity}) by doing a computation.
\begin{align*}
    F^1_{\delta}(g_t\Lambda)&=\min\left\{u: \begin{bmatrix}
        u\\ v
    \end{bmatrix} \in g_t\Lambda\cap C_\delta  \right\}\\
    &=\min\left\{u: g_{-t}\begin{bmatrix}
        u\\ v
    \end{bmatrix} \in \Lambda\cap g_{-t}C_\delta  \right\}\\
    &=\min\left\{u: \begin{bmatrix}
        e^{-\frac{t}{2}}u\\ e^{\frac{t}{2}}v
    \end{bmatrix} \in \Lambda\cap C_{\delta e^t}  \right\}\\
    &=e^{\frac{t}{2}}\min\left\{u': \begin{bmatrix}
        u'\\ v'
    \end{bmatrix} \in \Lambda\cap C_{\delta e^{t}}  \right\}\\
    &=e^{\frac{t}{2}}F^1_{\delta e^{t}}(\Lambda).
\end{align*}
This completes the proof of \eqref{geohomogeneity}.

We now proceed to prove equation (\ref{Fdeltaqminrelation}). We begin by defining a correspondence between rational numbers and lattice points. If $\frac{p}{q}$ is a rational number written in simplest form, we identify it with the integer vector $[q,p]^T$. We claim that $\frac{p}{q}\in (x-\delta,x+\delta)$ precisely when $h_x[q,p]^T\in C_\delta$.

\emph{Proof of claim:} 
Notice that $h_x[q,p]^T=[q,p-qx]^T$. Then we have that $[q,p-qx]^T\in C_\delta$ precisely when $\delta>\abs{\frac{p-qx}{q}}=\abs{x-\frac{p}{q}}$ which is equivalent to $\frac{p}{q}\in (x-\delta,x+\delta)$.
This completes the proof of our claim.

Our claim implies that the set of denominators of fractions in $(x-\delta,x+\delta)$ is precisely the set of $x$-coordinates of the lattice points in $h_x\Z^2\cap C_\delta$. In particular they have the same minimum, which means $q_{\min}(x,\delta)=F^1_{\delta}(h_x\Z^2)$ as desired.

\end{proof}
\begin{figure}
\centering
\includegraphics[width=100mm]{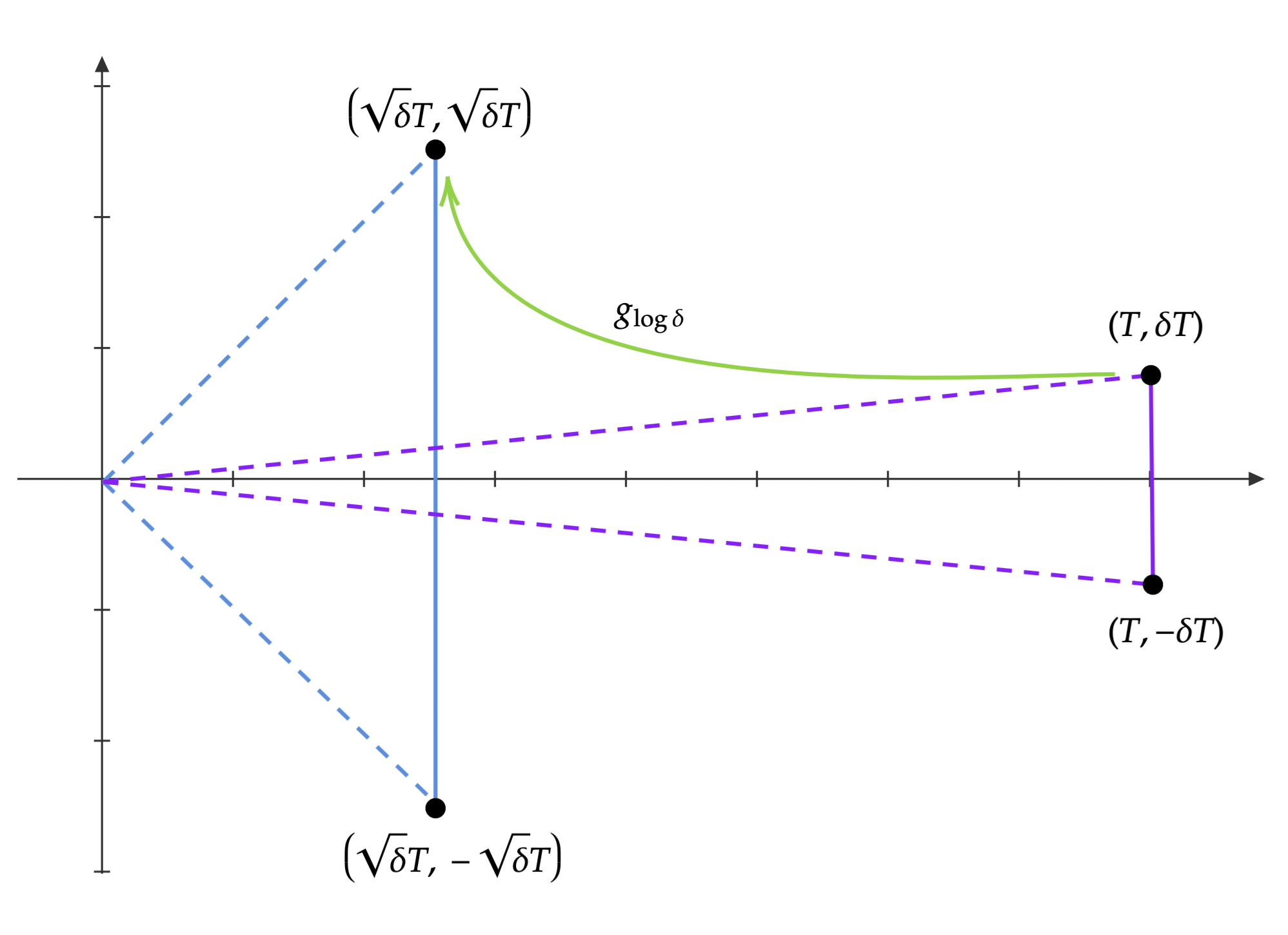}
\caption{Action of $g_{\log \delta}$ on $C_\delta \cap \{(x,y):x\leq T\}$. }\label{gtcone}
\end{figure}

Equation (\ref{geohomogeneity}) in particular allows us understand how the quantity $F^1_\delta(\Lambda)$ changes as we change $\Lambda$ in relation to the geodesic flow. Notice that if $t=-\log\delta$, we have that
\begin{equation}\label{homogeneity}
    F^1_\delta(g_{-\log\delta}\Lambda)= \delta^{-\frac{1}{2}}F^1_1(\Lambda) 
\end{equation}
for each $\Lambda \in X_2$. The identity (\ref{homogeneity}) allows us to exchange the problem of understanding $q_{\min}$ as a problem with a fixed lattice and changing region of intersection with a randomly selected lattice intersecting a fixed region of $\R^2$.
Figure~\ref{gtcone} contains a visual representation of the effect of $g_{\log\delta}$ on the cone $C_\delta$.



\begingroup
\def\thetheorem{\ref{cumulative 2}}
\begin{theorem}
Let $P$ denote the uniform probability measure on $[0,1]$. Then for every $T>0$, as $\delta\rightarrow 0$,
\begin{equation}
   P\left(\left\{x: \sqrt{\delta}q_{\min}(x,\delta)\leq T \right\}\right)\rightarrow \mu_2\left(\left\{\Lambda\in X_2: F^1_1(\Lambda)\leq T\right\}\right). 
\end{equation}
\end{theorem}
\addtocounter{theorem}{-1}
\endgroup

\begin{proof}
By equation (\ref{Fdeltaqminrelation}),
$\sqrt{\delta}q_{\min}(x,\delta)\leq T$ precisely when $\sqrt{\delta}F^1_\delta(h_x\Z^2)\leq T$. This means that $$P\left(\left\{ x: \sqrt{\delta}q_{\min}(x,\delta)\leq T \right\}\right)=P\left(\left\{ x: \sqrt{\delta}F^1_\delta(h_x\Z^2)\leq T \right\}\right).$$
Using equation (\ref{homogeneity}), we get that
\begin{align*}
    P\left(\left\{ x: \sqrt{\delta}F^1_\delta(h_x\Z^2)\leq T \right\}\right)&=P\left(\left\{ x: \sqrt{\delta}F^1_\delta(g_{-\log\delta} g_{\log\delta}h_x\Z^2)\leq T \right\}\right)\\
    &=P\left(\left\{ x: F^1_1(g_{\log\delta}h_x\Z^2)\leq T \right\}\right).
\end{align*}
Notice that $g_{\log\delta}h_x=h_{\frac{x}{\delta}}g_{\log{\delta}}$.
This means that the lattice $g_{\log\delta}\Z^2$ has period $\delta^{-1}$ under the horocyclic flow. Hence, by Theorem~\ref{equidistribution}, we have that as $\delta\rightarrow 0$,
$$P\left(\left\{ x: F^1_1(g_{\log\delta}h_x\Z^2)\leq T \right\}\right)\rightarrow\mu_2\left(\left\{\Lambda: F^1_1(\Lambda)\leq T\right\}\right)$$
as desired.

\end{proof}

Theorem~\ref{cumulative 2} immediately gives us the following corollary which provides a geometric interpretation of the scalar found in equation (\ref{CH}) as the solution to the integral below.
\begin{corollary}\label{MeanF1} As $\delta\rightarrow0$,
$$\sqrt{\delta}\mathbb{E}_x[q_{\min}(x,\delta)]=\sqrt{\delta}\int_0^1 q_{\min}(x,\delta) dP(x)\rightarrow\int_{X_2}F^1_1(\Lambda)d\mu_2(\Lambda).$$
\end{corollary}
\begin{proof}
    Notice that
    \begin{align*}
        \sqrt{\delta}\mathbb{E}_{x}[q_{\min}(x,\delta)]&=\int_0^1\sqrt{\delta}q_{\min}(x,\delta)dP(x)\\
        &=\int_0^\infty P\left(\{x\in[0,1]:\sqrt{\delta}q_{\min}(x,\delta)\geq T\}\right)dT.
    \end{align*}
By Theorem~\ref{cumulative 2}, we have that $$\sqrt{\delta}\mathbb{E}_x[q_{\min}(x,\delta)]\rightarrow \int_{X_2}F^1_1(\Lambda)d\mu_2(\Lambda),$$
as desired.
\end{proof}
\begin{corollary}
    $$\int_{X_2} F^1_1(\Lambda)d\mu_2(\Lambda)=\frac{16}{\pi^2}.$$
\end{corollary}

\begin{proof}
    This is a consequence of Corollary~\ref{MeanF1} and Theorem~\ref{theorem:CH}.
\begin{align*}
    \int_{X_2}F^1_1(\Lambda)d\mu_2(\Lambda)&=\lim_{\delta\rightarrow0} \sqrt{\delta}\int_0^1 q_{\min}(x,\delta) dx\\
    &=\lim_{\delta\rightarrow0}\sqrt{\delta}\left(\frac{16}{\pi^2}\frac{1}{\sqrt{\delta}}+O\left(\log^2(\delta)\right)\right)\\
    &= \frac{16}{\pi^2}
\end{align*}   

\end{proof}

\section{Equivariant processes}\label{Siegel Measures}

Before proceeding to generalize the minimal denominator function, we describe a general framework for equidistribution theorems. This circle of ideas have been inspired by the work of Marklof-Str\"ombergsson~\cite{MarklofStrombergsson} and Veech~\cite{Veech}, who in part was inspired by the work of Siegel~\cite{Siegel} and Masur~\cite{EskinMasur}. We take the setup as described in Athreya-Ghosh~\cite{AthreyaGhosh}. Let $m\geq 2$ and $G\subset GL(m,\R)$ be a subgroup. Suppose $(X,\lambda)$ is a standard Borel space and $G$ acts on $X$ via measure-preserving transformations. A ($G-$)\emph{equivariant process}, also known as a \emph{Siegel measure}, is a triple $(X,\mu,\nu)$ where $\nu$ is a map
$\nu:X\rightarrow \mathcal{M}(\R^m)$,
where $\mathcal{M}(\R^m)$ is the space of $\sigma$-finite Radon Borel measures on $\R^m$ and we have that for each $g\in G$ and $x\in X$, $\nu(gx)=g_*\nu(x)$. We shall call $\nu$ the \emph{equivariant process map.}

\subsection{Chen-Haynes distributions}
Let $\mathcal{S}=\{S_T\}_{T>0}$ be a family of Borel subsets of $\R^m$ with the property that if $T_1\leq T_2$, then $S_{T_1}\subset S_{T_2}$. Let $(X_n,\lambda_n,\nu_n)$ be a sequence of equivariant processes. We define the \emph{Chen-Haynes distribution} associated to $\mathcal{S}$ and $(X_n,\lambda_n,\nu_n)$ by

\begin{equation}\label{chen-haynes distribution}
    \xi(X_n,\lambda_n,\nu_n,\mathcal{S})(T)=\lim_{n\rightarrow \infty} \lambda_n\left(\left\{x\in X_n: \nu_n(x)(S_T)\geq 1\right\}\right).
\end{equation}

\paragraph*{\bf Equidistribution} We specialize the setup above by looking at sequences of $G$-equivariant processes over a fixed space and a fixed equivariant process map. Let $(X,\lambda_n,\nu)$ be such a sequence of $G$-equivariant processes. Then we have the following result which allows us to exchange weak$^*$ convergence results for equidistribution results.

\begin{theorem}\label{general equidistribution}
    Suppose that $\lambda(\partial\{x\in X: \nu(x)(S_T)\geq 1\})=0$. If $\lambda_n\stackrel{\ast}{\rightharpoonup} \lambda$, then
    \begin{equation}\label{concrete Chen-Haynes Distribution}
        \xi(X,\lambda_n,\nu,\mathcal{S})(T)=\lambda\left(\left\{x\in X: \nu(x)(S_T)\geq 1\right\}\right).
    \end{equation}
\end{theorem}
\begin{proof}
    Since all of our measures are Radon Borel, we have that weak$^*$ convergence of measures implies that for all bounded continuous functions $f:X\rightarrow \R$, $\lambda_n(f)\rightarrow\lambda(f).$ This means that if $A$ is a set subset of $X$ with $\lambda(\partial A)=0$, then we may approximate $\chi_A$ from above and below via bounded continuous functions. This means that $\lambda_n(A)\rightarrow \lambda(A)$.
    We complete the proof by setting $A= \{x\in X: \nu(x)(S_T)\geq 1\}$.

\end{proof}

\paragraph*{\bf Theorem~\ref{cumulative 2} revisited}
In the case of Theorem~\ref{cumulative 2}, we see that the sets $S_T=C_1\cap\{[x,y]^T\in \R^2: x\leq T\}$ and the measures $\lambda_n$ given by the uniform measure on the $h_s$-orbit of $g_{\log{1/n}}\Z^2$ in $X_2$ provides us with a sequence of standard Borel spaces. Let $G=\{h_s:s\in \R\}\subset SL(2,\R)$ and $\nu: X_2\rightarrow \mathcal{M}(\R^2)$ be given by
\begin{equation}\label{equivariant process map}
    \nu(\Lambda)=\sum_{x\in \Lambda}\delta_x,
\end{equation}
\noindent
where $\delta_x$ is the Dirac-delta measure with support $\{x\}$. This construction provides us with a sequence of equivariant processes $(X_2,\lambda_n, \nu)$ 

Notice that $F_1^1(\Lambda)\leq T$ precisely when $\nu(\Lambda)(S_T)\geq 1$. This rephrases Theorem~\ref{cumulative 2} as the computation of the Chen-Haynes distribution associated to the sequence
$(X_2,\lambda_n, \nu)$ and the family of cones $\mathcal{S}$.
\begin{align*}
    \xi(X_2,\lambda_n,\nu,\mathcal{S})(T)&=\lim_{n\rightarrow \infty} \lambda_n\left(\left\{\Lambda\in X_2: \nu(\Lambda)(S_T)\geq 1\right\}\right)\\
    &=\lim_{n\rightarrow \infty} \lambda_n\left(\left\{\Lambda\in X_2: F_1^1(\Lambda)\leq T\right\}\right)\\
    &=\mu_2\left(\left\{\Lambda\in X_2: F_1^1(\Lambda)\leq T\right\}\right).
\end{align*}
\noindent
In what follows, we will compute the Chen-Haynes distribution of equivariant process associated to higher dimensional Diophantine approximations and the holonomy vectors of translation surfaces. We will proceed similarly, providing appropriate families of sets $\mathcal{S}$, sequences of measure $\{\lambda_n\}$ in the pertinent space, and weak$^*$ convergence results in order to satisfy the hypotheses of Theorem~\ref{general equidistribution}. The equivariant process map will always be of the form seen in equation (\ref{equivariant process map}).







\section{Generalization of $q_{\min}$ in higher dimensions}\label{higherdimensions}

\noindent The function $q_{\min}$ arose from a question about minimizing denominators in a $\delta$ neighborhood of a randomly chosen point in $[0,1]$. A natural generalization of this question is the the following:
\newline

\noindent\textbf{Question:} Let $m$ be a positive integer and endow $\R^m$ with the max norm, $\norm{\cdot}_m$. Pick $\delta>0$ and $\mathbf{x}\in [0,1]^m$. We ask, what is the smallest positive integer $q$ such that there is a vector $\mathbf{p}\in \Z^m$ such that $\norm{\mathbf{x}-\frac{1}{q}\mathbf{p}}_m<\delta?$
\newline
\newline
\noindent This question gives a natural extension of the $q_{\min}$ function: For $m\in \N$, $\textbf{x}\in [0,1]^m$ and $\delta>0$. We define 
$$Q^m(\textbf{x},\delta)=\min\left\{q\in\N:\text{ there exists }\textbf{p}\in \Z^m \text{ such that }\norm{\textbf{x}-\frac{1}{q}\textbf{p}}_m<\delta\right\}.$$

\noindent Motivated by the work in \S\ref{existence of limiting distribution}, we study the statistics of $Q^m(\textbf{x},\delta)$ by relating it to the theory of lattices. 


Define the following function
$F^m_\delta:X_{m+1}\rightarrow\N$ by $$F^m_\delta(\Lambda)=\min\left\{u: \text{there exists } \textbf{v}\in \R^m \text{ such that }\begin{bmatrix}
    u\\ \textbf{v}
\end{bmatrix}\in \Lambda \cap C^m_\delta\right\},$$
where $C^m_\delta=\left\{\begin{bmatrix}t\\t\mathbf{x}\end{bmatrix}: \norm{\textbf{x}}_{m}<\delta, t>0\right\}$.
With these new definitions, we are now able to compute the limiting cumulative distribution for a properly normalized version of $Q^m(\textbf{x},\delta)$, that being $\delta^{\frac{m}{m+1}}Q^m(x,\delta)$.

\begin{theorem} Let $P$ denote the uniform distribution on $[0,1]^m$. Then, as $\delta\rightarrow 0$,
    $$P\left(\left\{\mathbf{x}\in [0,1]^m:\delta^{\frac{m}{m+1}}Q^m(\mathbf{x},\delta)\leq T\right\}\right)\rightarrow \mu_{m+1}\left(\left\{\Lambda\in X_{m+1}: F^{m}_1(\Lambda)\leq T\right\}\right).$$
\end{theorem}
\begin{corollary} As $\delta\rightarrow 0$,
    $$\mathbb{E}_{\mathbf{x}}\left[\delta^{\frac{m}{m+1}}Q^{m}(\mathbf{x},\delta)\right]=\delta^{\frac{m}{m+1}}\int_{[0,1]^m}Q^m(\mathbf{x},\delta)dP(\mathbf{x})\rightarrow \int_{X_{m+1}}F_1^{m}(\Lambda)d\mu_{m+1}(\Lambda).$$
\end{corollary}
We omit the proof of Theorem \ref{cumulative m+1} as it will be a consequence of Theorem~\ref{linear transformations}. 

\subsection{Linear Forms}
Let $m$ and $n$ be positive integers. Endow $\R^m$ and $\R^n$ with their max norms $\norm{\cdot}_m$ and $\norm{\cdot}_n$, respectively. The norms $\norm{\cdot}_m$ and $\norm{\cdot}_n$ induce a norm, $\norm{\cdot}$, on $\R^m\oplus \R^n=\R^{m+n}$ given by $\norm{[\mathbf{u},\mathbf{v}]^T}=\norm{\mathbf{u}}_n+\norm{\mathbf{v}}_m$. Let $X$ be an $n\times m$ matrix with entries in $[0,1]$.
For $\delta>0$, define
$$Q^{m,n}(X,\delta)= \min_{(\mathbf{q},\mathbf{p})\in\Z^n\times\Z^m}\left\{\norm{\mathbf{q}}_n:\norm{X\mathbf{q}-\mathbf{p}}_m<\delta\norm{\mathbf{q}}_n\right\}.$$
Notice that the case when $n=1$, $Q^{m,1}(\mathbf{x},\delta)=Q^m(\mathbf{x},\delta)$ and when $m$ and $n$ are both $1$, $Q^{1,1}(x,\delta)=q_{\min}(x,\delta)$.

We may identify the collection of $m\times n$ matrices with entries in $[0,1]$ with $[0,1]^{mn}$. $Q^{m,n}(X,\delta)$ then models the question: Given a randomly selected $X\in [0,1]^{mn}$, what is the shortest integer vector $\mathbf{q} \in \Z^n$ that lands within an appropriately sized neighborhood of $\Z^m$.
The size of this neighborhood depends itself on the size of $\mathbf{q}$.

In order to study the statistics of $Q^{m,n}$ we will relate it to the theory of lattices just as we did in \S\ref{existence of limiting distribution} and give appropriate generalizations of the geodesic and horocyclic flows.

\subsection{Lattice Interpretation}
%
%

Define $F_\delta^{m,n}:X_{m+n}\rightarrow \R$ by
$$F_\delta^{m,n}(\Lambda)=\min\left\{\norm{\mathbf{u}}_n: \begin{bmatrix}\mathbf{u}\\ \mathbf{v}\end{bmatrix}\in \Lambda\cap C^{m,n}_\delta\right\}$$
where
$$C^{m,n}_\delta=\left\{r\begin{bmatrix}\mathbf{s}\\\mathbf{t}\end{bmatrix}: r>0, \norm{\mathbf{s}}_n=1,\norm{\mathbf{t}}_m<\delta\right\}.$$


\noindent The higher-dimensional versions of the horocyclic and geodesic flows that we will use are the following:
Let $m,n\geq 1$ and let $X\in [0,1]^{mn}$, define
\begin{align*}
h_X^{m,n}=\begin{bmatrix}\Id_n & 0\\
-X & \Id_m\end{bmatrix}&& \text{ and} &&
g_t^{m,n}=\begin{bmatrix}e^{\frac{t}{m+n}}\Id_n & 0\\
0 & e^{\frac{-nt}{m(m+n)}}\Id_m\end{bmatrix}. 
\end{align*}

The next lemma is a higher dimensional analogue of Lemma~\ref{good lemma}.

\begin{Lemma}\label{lemma m,n} With the notation as above,

\begin{equation}\label{m,n homogeneity}
        F_{\delta}^{m,n}(g^{m,n}_t\Lambda)=e^{\frac{t}{m+n}}F_{\delta e^{\frac{t}{m}}}(\Lambda).
    \end{equation}
and
    \begin{equation}\label{Qmn Fmn relation}
        Q^{m,n}(X,\delta)=F_\delta^{m,n}(h^{m,n}_X\Z^{m+n}).
    \end{equation}
\end{Lemma}
\begin{proof} We begin with the proof of equation (\ref{m,n homogeneity}). We first explore the effects of the geodesic flow on the cone $C^{m,n}_\delta$. Just as in Lemma~\ref{good lemma}, we have that $g_t^{m,n}$ expands our cone in the the direction of $\R^m$ by $e^{\frac{t}{m}}$. More precisely, we have 
\begin{equation}
    g_t^{m,n}C^{m,n}_\delta=C^{m,n}_{\delta e^{-\frac{t}{m}}}
\end{equation}

Hence, we have the following computation:
\begin{align*}
    F_\delta^{m,n}(g_t^{m,n}\Lambda)&=\min\left\{\norm{\mathbf{u}}_n:\begin{bmatrix}
        \mathbf{u}\\ \mathbf{v}
    \end{bmatrix}\in g_t^{m,n}\Lambda \cap C_\delta^{m,n}\right\}\\
    &=\min\left\{\norm{\mathbf{u}}_n:g_{-t}^{m,n}\begin{bmatrix}
        \mathbf{u}\\ \mathbf{v}
    \end{bmatrix}\in\Lambda \cap g_{-t}^{m,n}C_\delta^{m,n}\right\}\\
    &=\min\left\{\norm{\mathbf{u}}_n:\begin{bmatrix}
        e^{-\frac{t}{m+n}}\mathbf{u}\\ e^{-\frac{nt}{m(m+n)}}\mathbf{v}
    \end{bmatrix}\in\Lambda \cap C_{\delta e^{\frac{t}{m}}}^{m,n}\right\}\\
    &=e^{\frac{t}{m+n}}\min\left\{\norm{\mathbf{u'}}_n:\begin{bmatrix}
        \mathbf{u'}\\ \mathbf{v'}
    \end{bmatrix}\in\Lambda \cap C_{\delta e^{\frac{t}{m}}}^{m,n}\right\}\\
    &= e^{\frac{t}{m+n}}F^{m,n}_{\delta e^{\frac{t}{m}}}(\Lambda)
\end{align*}
as desired.

We now proceed to prove equation (\ref{Qmn Fmn relation}).
We first show that $Q^{m,n}(X,\delta)\geq F_\delta^{m,n}(h^{m,n}_X\Z^{m+n})$.
Let $(\mathbf{q},\mathbf{p})\in \Z^{n}\times \Z^m$ such that $\norm{X\mathbf{q}-\mathbf{p}}_m<\delta\norm{\mathbf{q}}_n$
and $\norm{\mathbf{q}}_n$ is minimized.
Then by definition, $Q^{m,n}(X,\delta)= \norm{\mathbf{q}}_n$. We then have that $$h_X\begin{bmatrix}
    \mathbf{q}\\ \mathbf{p}
\end{bmatrix}=\begin{bmatrix}
    \mathbf{q} \\ \mathbf{p}-X\mathbf{q}
\end{bmatrix}\in h_X\Z^{m+n}\cap C_\delta^{m,n}
$$
Hence we have that
$$F^{m,n}_\delta(h_X\Z^{m+})\leq \norm{\mathbf{q}}_n=Q_\delta^{m,n}(X,\delta).$$

Next we show that $F^{m,n}_\delta(h_X\Z^{m+n})\geq Q^{m,n}(X,\delta).$
Let $(\mathbf{a},\mathbf{b})\in h_X\Z^{m+n}\cap C_\delta^{m,n}$ such that $\norm{\mathbf{a}}_n$ is minimized. This means $F_\delta^{m,n}(h_X\Z^{m+n})=\norm{\mathbf{a}}_n$.
This means that there exist $(\mathbf{q},\mathbf{p})\in\Z^{n}\oplus\Z^{m}$ such that $$h_X\begin{bmatrix}
    \mathbf{q}\\ \mathbf{p}
\end{bmatrix}=\begin{bmatrix}
    \mathbf{a}\\ \mathbf{b}
\end{bmatrix}$$
That is $\mathbf{a}=\mathbf{q}$ and $\mathbf{p}-X\mathbf{q}=\mathbf{b}$.
Since $$\begin{bmatrix}
    \mathbf{a}\\ \mathbf{b}
\end{bmatrix}=\norm{\mathbf{a}}_n\begin{bmatrix}
    \frac{\mathbf{a}}{\norm{\mathbf{a}}_n}\\ \frac{\mathbf{b}}{\norm{\mathbf{a}}_n}
\end{bmatrix}\in C_\delta^{m,n}$$
it follows that $\norm{\frac{\mathbf{b}}{\norm{\mathbf{a}}_n}}_m<\delta$ which means
$\norm{\mathbf{b}}_m<\delta\norm{\mathbf{a}}_n$. This can be rewritten as $\norm{\mathbf{p}-X\mathbf{q}}_m<\delta\norm{\mathbf{q}}_n$. That is 
$$Q^{m,n}(X,\delta)\leq \norm{\mathbf{q}}_n=\norm{\mathbf{a}}_n=F_\delta^{m,n}(h_X\Z^{m+n}).$$
This completes the proof of Lemma~\ref{lemma m,n}.
\end{proof}

Notice that when $t=-\log(\delta^m)$, equation (\ref{m,n homogeneity}) says that
\begin{equation}\label{highergeohomogeneity}
    F_\delta^{m,n}(g^{m,n}_{-\log(\delta^m)}\Lambda)=\delta^{-\frac{m}{m+n}}F_1^{m,n}(\Lambda).
\end{equation}

Before proceeding to state the main theorem of this section we state an equidistribition theorem due to Kleinbock and Margulis~\cite{KleinMarg} which will play the roll Theorem~\ref{equidistribution} played in the proof of Theorem~\ref{cumulative 2}.
\newline
\paragraph*{\bf Almost uniformly continuous functions.} Let $(Y,\mu)$ be a topological space equipped with its Borel $\sigma$-algebra and a measure $\mu$. We say a function $\Psi: Y\rightarrow \R$ is \emph{almost uniformly continuous} if there exist sequences $\{\Psi_i\}_{i\in \N}$ and $\{\Psi^j\}_{j\in\N}$ of uniformly continuous functions on $Y$ such that $\Psi_i$ increases almost surely to $\Psi$ and $\Psi_j$ decreases almost surely to $\Psi$. In particular, with the setup above, if $\mu$ is a regular probability measure and $A$ is a measurable subset of $Y$ with $\mu(\partial A)=0$, then the indicator function of $A$, $\chi_A$, is almost uniformly continuous.

\begin{Lemma}\textbf{(Kleinbock-Margulis~\cite{KleinMarg})}
\label{KleibockMargulis}
Let $f\in L^2(M_{m\times n}(\R))$ with compact support. Then for any almost uniformly continuous $\Psi\in L^2(X_{m+n})$, any compact subset L of $X_{m+n}$, and any $\eps>0$, the exists $T>0$ such that
\begin{equation}
 \abs{\int_{M_{m\times n}(\R)}f(h_X)\Psi(g_th_X\Lambda)dX-\int_{M_{m\times n}(\R)}f(h_X)dX \int_{X_{m+n}}\Psi(\Lambda)d\mu_{m+n}(\Lambda)}<\eps
\end{equation}
for any $\Lambda\in L$ and all $t>T$.
\end{Lemma}

Before proving Theorem~\ref{linear transformations}, we give some notation and a lemma which will be useful in the proof.

    \begin{enumerate}
        \item $A_T=\{\Lambda\in X_{m+n}: F_1^{m,n}(\Lambda)\leq T$\},
        \item $B_T=X_{m+n}\setminus A_T$,
        \item $C^R=\overline{C_1^{m,n}}\cap \{[\mathbf{u},\mathbf{v}]^T\in \R^n\oplus\R^m: \norm{\mathbf{u}}_n\leq R\},$
        \item For each $\Lambda\in X_{m+n}$, $\Lambda^*=\Lambda\setminus \{\mathbf{0}\}$.
    \end{enumerate}


\begin{Lemma}\label{Study of A_T}
   With the notation as above, the following are true:
   \begin{equation}\label{Closure of A_T}
       \overline{A_T}\subset A_T\cup \{\Lambda\in X_{m+n}: \Lambda^*\cap \partial C_1^{m,n}\neq \emptyset\}, 
   \end{equation}

   \begin{equation}\label{Boundary of A_T}
       \partial A_T\subset \{\Lambda\in X_{m+n}: F_1^{m,n}= T\}\cup \{\Lambda\in X_{m+n}: \Lambda^*\cap \partial C_1^{m,n}\neq \emptyset\}
   \end{equation}
   and
   \begin{equation}\label{measure of boundary of A_T}
       \mu_{m+n}(\partial A_T)=0.
   \end{equation}
\end{Lemma}

\begin{proof}
We begin by proving equation 
(\ref{Closure of A_T}). Fix $T>0$. Let $\Lambda\in \overline{A_T}$. Suppose $\Lambda\in \overline{A_T}$ and 
\newline
$\Lambda\notin\left\{\Lambda\in X_{m+n}: \Lambda^*\cap \partial C_1^{m,n}\neq \emptyset\right\}$. If $F_1^{m,n}(\Lambda)\leq T$, we are done, so suppose $F_1^{m,n}(\Lambda)=T_0>T$. Let $B$ be the closed ball in $\R^{m+n}$ centered at the origin with radius $R$, where $R$ is chosen such that $C^{T_0}\subset \operatorname{Int}(B)$ and $\partial B\cap \Lambda=\emptyset$. Since $B$ is compact, $\Lambda$ is discrete, and $F_1^{m,n}(\Lambda)=T_0<R$, it follows that $B\cap \Lambda$ is finite and non-empty. This implies that there exists $\delta>0$ such that if $[\mathbf{u},\mathbf{v}]^T\in B\cap\Lambda$ with $\norm{\mathbf{u}}_n\in (T_0-\delta,T_0+\delta)$, then $\norm{\mathbf{u}}_n=T_0$. Let $d=\operatorname{dist}(\partial B, \Lambda)$. $d>0$ by our choice of $R$. Define $\eta=\frac{1}{2}\min\{d,\delta,1,T_0-T\}$. Consider $B'$ to be the ball centered at the origin with radius $2R$. Since $\Lambda\cap B'$ is finite and the action of $SL(m+n,\R)$ on $\R^{m+n}$ is continuous, there exists a neighborhood $U$ of $\Id\in SL(m+n,\R)$ such that if $\mathbf{w}\in B'\cap \Lambda$ and $g\in U$, $\norm{g\mathbf{w}-\mathbf{w}}<\eta$. Since $\Lambda\in \overline{A_T}$, there exists a sequence $\Lambda_i\in A_T$ converging to $\Lambda$. Since $\Lambda_i$ is converging to $\Lambda$, we can write $\Lambda_i=g_i\Lambda$ where $g_i\in U$ for $i$ large enough. Then this implies that
$$\abs{T_0-T}\leq\abs{F_1^{m,n}(\Lambda)-F_1^{m,n}(\Lambda_i)}< \eta<\abs{T_0-T}.$$
This is a contradiction. This means our assumption that $F_1^{m,n}(\Lambda)=T_0>T$ is false. Hence, $F_1^{m,n}(\Lambda)\leq T$, so $\Lambda\in A_T$. This completes the proof of equation (\ref{Closure of A_T}).

We now proceed to prove equation (\ref{Boundary of A_T}). One can use the a similar argument as the one used in the proof of equation (\ref{Closure of A_T}) to show that \begin{equation*}\label{Closure of B_T}
    \overline{B_T}\subset B_T\cup \{\Lambda\in X_{m+n}: \Lambda^*\cap \partial C_1^{m,n}\neq \emptyset\}.
\end{equation*}
This then implies
\begin{equation}
    \partial A_T=\overline{A_T}\cap \overline{B_T}\subset \left\{\Lambda\in X_{m+n}: F_1^{m,n}(\Lambda)= T\right\}\cup\left\{\Lambda\in X_{m+n}: \Lambda^*\cap \partial C_1^{m,n}\neq \emptyset\right\}
\end{equation}

Finally, we prove equation (\ref{measure of boundary of A_T}). Let $D=\left\{[\mathbf{u},\mathbf{v}]^T\in \R^{m+n}: \norm{\mathbf{u}}_n=T\right\}\cup \partial C_1^{m,n}.$ $D$ has Lebesgue measure $0$ on $\R^{m
+n}$, this then implies that $\{\Lambda\in X_{m+n}: \Lambda^*\cap D\neq\emptyset\}$ has measure $0$. Since 
    \begin{align}\label{containment of boundary of A_T}
        \partial A_T&=\overline{A_T}\cap \overline{B_T}\subset \left\{\Lambda\in X_{m+n}: F_1^{m,n}(\Lambda)= T\right\}\cup\left\{\Lambda\in X_{m+n}: \Lambda^*\cap \partial C_1^{m,n}\neq \emptyset\right\}\\
        &=\{\Lambda\in X_{m+n}: \Lambda^*\cap D\neq\emptyset\},
    \end{align}

we have that $\partial A_T$ has measure zero.

\end{proof}

We are now ready to state and prove Theorem~\ref{linear transformations}
\begin{theorem} \label{linear transformations}
    Let $P$ be the uniform probability measure on $[0,1]^{mn}$. Then as $\delta\rightarrow 0$,
    $$P\left(\left\{X\in M_{m\times n}([0,1]):\delta^{\frac{m}{m+n}}Q^{m,n}(X,\delta)\leq T\right\}\right)\rightarrow \mu_{m+n}\left(\left\{ \Lambda\in X_{m+n}: F_1^{m,n}(\Lambda)\leq T\right\}\right)$$
\end{theorem}


%
%
%
%

\begin{proof}

Equation (\ref{Qmn Fmn relation}), allows us to interchange
$P\left(\left\{X\in M_{m\times n}([0,1]):\delta^{\frac{m}{m+n}}Q^{m,n}(X,\delta)\leq T\right\}\right)$ for $P\left(\left\{X\in M_{m\times n}([0,1]):\delta^{\frac{m}{m+n}}F^{m,n}_\delta(h_X\Z^{m+n})\leq T\right\}\right).$
By equation (\ref{measure of boundary of A_T}) we get that $\chi_{A_T}$ is almost uniformly continuous.

Using equation (\ref{highergeohomogeneity}), we get that 
\begin{align*}
   \textcolor{white}{fds}& P\left(\left\{X\in M_{m\times n}([0,1]) :\delta^{\frac{m}{m+n}} F^{m,n}_\delta(h_X\Z^{m+n})\leq T\right\}\right) \\ 
   &\hspace{0.3in}=P\left(\left\{X\in M_{m\times n}([0,1]):F^{m,n}_\delta(g_{-\log(\delta^m)}h_X\Z^{m+n})\leq T\right\}\right)\\ &\hspace{0.3in}=\int_{[0,1]^{mn}}\chi_{A_T}\left(g_{-\log(\delta^m)}h_X\Z^{m+n}\right)dP(X). 
\end{align*}
We can then apply Lemma~\ref{KleibockMargulis} by setting $f=\chi_{[0,1]^{mn} }$ and  $\Psi= \chi_{A_T}$.

$$\lim_{\delta\rightarrow 0}\int_{[0,1]^{mn}}\chi_{[0,1]^{mn}}(X)\chi_A\left(g_{-\log(\delta^m)}h_X\Z^{m+n}\right)dP(X)=\int_{X_{m+n}}\chi_{A_T}(\Lambda)d\mu_{m+n}(\Lambda)=\mu(A_T).$$
as claimed.
\end{proof}
\begin{corollary}
    $$\lim_{\delta\rightarrow 0}\mathbb{E}_X\left[\delta^{\frac{m}{m+n}}Q^{m,n}(X,\delta)\right]=\int_{X_{m+n}}F_1^{m,n}(\Lambda)d\mu_{m+n}(\Lambda).$$
\end{corollary}

\section{Short saddle connections}\label{Saddle Connections}
\subsection{The $SL(2,\R)$ action on the Moduli Space of Translation Surfaces and Veech Surfaces}
Given $g>0$ and an integer partition $\alpha$ of $2g-2$, we define $\mathcal{H}(\alpha)$ to be the moduli space of translation surfaces $\omega$ with genus $g$, area 1, and zeros with orders given by $\alpha$. The space $\mathcal{H}(\alpha)$ has a natural topology and finite Borel measure $\mu_{MSV}$. Zorich~\cite{Zorich} contains a detailed discussion of this structure on $\mathcal{H}(\alpha)$.

Let $s_1$ and $s_2$ be two parallel segments of the same length on $\C$. If $g\in SL(2,\R)$, then it follows that $g(s_1)$ and $g(s_2)$ are also two parallel segments of the same length on $\C$. In particular, given that we may think of translation surfaces as polygons on the plane with identifications along its sides via Euclidean translation, it follows that the action of $SL(2,\R)$ on $\C$ transfers to an action on our translation surface.
The action of $SL(2,\R)$ is continuous and ergodic on each connected component of $\mathcal{H}(\alpha)$. 

In this section we are particularly interested in translation surfaces which exhibit a large number of symmetries. To be precise, we provide the following definitions.
\begin{Definition}
    Let $\omega$ be a translation surface. The Veech group of $\omega$ is the stabilizer of $\omega$ under the $SL(2,\R)$ action.
\end{Definition}
\begin{Definition}
    We say $\omega$ is a Veech surface if its Veech group is a discrete subgroup of $SL(2,\R)$, where $SL(2,\R)/\Gamma_\omega$ has finite volume.
\end{Definition}

For simplicity we will denote $SL(2,\R)/\Gamma_\omega$ by $Y_\omega$.
In particular, $Y_\omega$ parameterizes the orbit of $\omega$ under $SL(2,\R)$. It turns out the the quotient space $Y_\omega$ is never compact. With this in mind, we state an important theorem due to Dani-Smillie~\cite{DaniSmillie}.

\begin{theorem}\label{VeechEqui}\textbf{(Dani-Smillie~\cite{DaniSmillie})}
    Let $\omega$ be a Veech surface and $(\omega_i)_{i\in\N}$ be a sequence of points in $Y_\omega$. Suppose $\omega_i$ has period $s_i$ under the the action of the horocyclic flow. Let $\nu_i$ be the uniform measure on the orbit of $\omega_i$. If $s_i\rightarrow \infty$, then
    $$\nu_i\stackrel{\ast}{\rightharpoonup} \frac{1}{\mu_{\omega}(Y_\omega)}\mu_\omega,$$
    where $\mu_\omega$ is the Haar measure on $Y_\omega$.
\end{theorem}

\subsection{Existence of limiting distribution for Veech surfaces} 
\begin{figure}[t!]
\centering
\includegraphics[width=90mm]{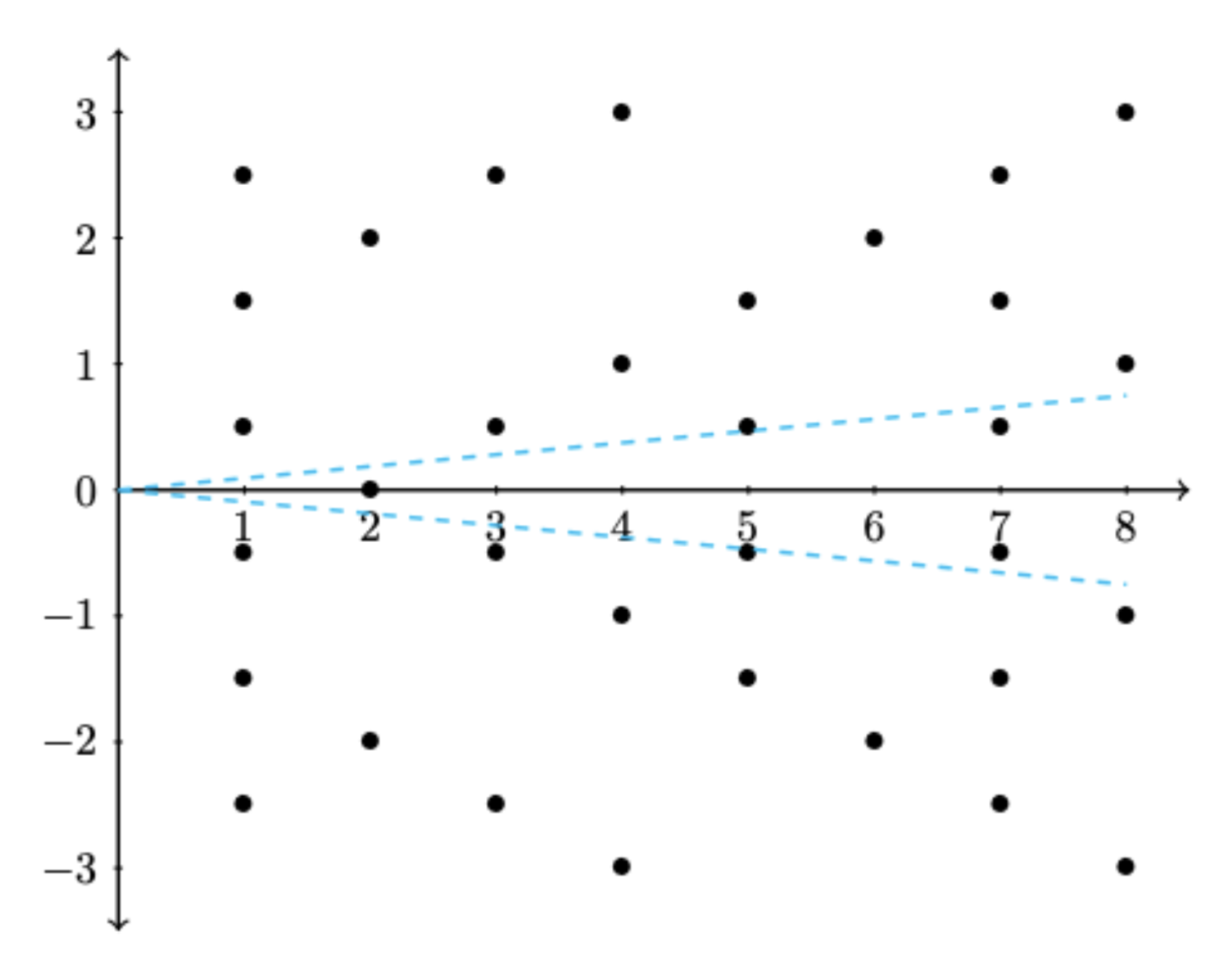}
\caption{Depiction of saddle connections on a genus 1 surface.} 
\label{Holonomy vectors intersecting a thin cone}
\end{figure}
A natural extension of the function $q_{\min}$ in the context of translation surfaces is the following:
\newline

\textbf{Question}: What is the point of smallest $x$-coordinate in $\Lambda_\omega\cap h_{-s}C_\delta$ as $s$ ranges over $\R$.
\newline

Due to the Veech dichotomy~\cite{HubSchm}, we know that in the context of Veech surfaces we can always find (up to rotating the surface if necessary the surface) an $\alpha>0$ such that $h_\alpha$ belongs to the Veech group of our surface.

Let $\delta>0$ and consider the following function: $$\Psi(\omega,\delta)=\min\left\{\Re(z_\gamma): \gamma \text{ is a saddle connection of } \omega \text{ and } z_\gamma\in \Lambda(\omega)\cap C_\delta\right\}.$$

\begin{Lemma}\label{lemma ts} Let $\omega$ be a translation surface, then

\begin{equation}
    \Psi(g_t\omega,\delta)=e^{\frac{t}{2}}\Psi(\omega,e^{t}\delta)
\end{equation}
and
\begin{equation}
    \Psi(g_{-\log\delta}\omega,\delta)=\sqrt{\delta}\Psi(\omega, 1).
\end{equation}

\end{Lemma}

\begingroup
\def\thetheorem{\ref{cummulative sc}}
\begin{theorem}
Let $\omega$ be a Veech surface and suppose that $h_{\alpha}\in \Gamma_\omega$ for some $\alpha>0$. Let $P$ be the uniform probability measure on $[0,\alpha]$. Then as $\delta\rightarrow 0$,
$$P\left(\left\{s\in [0,\alpha]:\sqrt{\delta}\Psi(h_s\omega,\delta)\leq T\right\}\right)\rightarrow \frac{\mu_{\omega}\left(\left\{ g\Gamma_\omega\in Y_\omega: \Psi(g\Gamma_\omega,1)\leq T\right\}\right)}{\mu_\omega\left(Y_\omega\right)}.$$

\end{theorem}
\addtocounter{theorem}{-1}
\endgroup
\begin{proof}
This proof has the same strategy as that of Theorem~\ref{cumulative 2}, we change to the appropriate equidistribution theorem to pass to the limit.
\newline

Let $$A= \{k\Gamma_\omega\in Y_\omega:\Psi(k\omega, 1)\leq T\}.$$
We now proceed to the computation
\begin{align*}
    P\left(\left\{s\in [0,\alpha]:\sqrt{\delta}\Psi(h_s\omega,\delta)\leq T\right\}\right)&=P\left(\left\{s\in [0,\alpha]:\sqrt{\delta}\Psi(g_{-\log \delta}g_{\log \delta}h_s\omega,\delta)\leq T\right\}\right)\\
    &=P\left(\left\{s\in [0,\alpha]:\Psi(g_{\log \delta}h_s\omega,1)\leq T\right\}\right)\\
    &=\frac{1}{\alpha}\int_0^\alpha \chi_{A}(g_{\log \delta}h_s\omega)ds
\end{align*}
Since $g_{\log \delta}h_s=h_{\frac{s}{\delta}}g_{\log\delta}$, we have that the period of $g_{\log{\delta}}\omega$ under the horocyclic flow is $\frac{\alpha}{\delta}$ and hence by Theorem~\ref{VeechEqui}, the orbit $g_{\log \delta}\omega$ under the horocyclic flow are becoming equidistributed as $\delta \rightarrow 0$. This means that
\begin{align*}
    \lim_{\delta\rightarrow0}\frac{1}{\alpha}\int_0^\alpha \chi_{A}(g_{\log \delta}h_s\omega)ds= \lim_{\delta\rightarrow0}\frac{\delta}{\alpha}\int_0^{\frac{\alpha}{\delta}}\chi_A(h_sg_{\log\delta}\omega)ds=\frac{1}{\mu_\omega(Y_\omega)}\int_{Y_\omega}\chi_A(g\Gamma_\omega)d\mu_\omega.
\end{align*}
And and we have that 
$$\frac{1}{\mu_\omega(Y_\omega)}\int_{Y_\omega}\chi_A(g\Gamma_\omega)d\mu_\omega=\frac{\mu_{\omega}\left(\left\{ g\Gamma_\omega\in Y_\omega: \Psi(g\Gamma_\omega,1)\leq T\right\}\right)}{\mu_\omega\left(Y_\omega\right)},$$
which is what we wanted to show.
\end{proof}
%
%
\begin{corollary}
    Let $\omega$ be a Veech surface and suppose that $h_{\alpha}\in SL(\omega)$ for some $\alpha>0$, then as $\delta\rightarrow 0$,
    $$\mathbb{E}_x\left[\sqrt{\delta}\Psi(h_s\omega,\delta)\right]=\frac{\sqrt{\delta}}{\alpha}\int_0^\alpha\Psi(h_s\omega,\delta)ds\rightarrow \frac{1}{\mu_\omega(X_\omega)}\int_{X_\omega}\Psi(g\Gamma_\omega,1) d\mu_{\omega}.$$
\end{corollary}

\end{document}